\documentclass[11pt,reqno]{amsart}

\usepackage{graphicx}
\usepackage{multicol,multirow}
\usepackage{amsmath,amssymb,amsfonts}
\usepackage{mathtools}
\usepackage{mathrsfs}
\usepackage{amsthm} 
\usepackage[numbers]{natbib}
\usepackage{ifpdf}
\usepackage[T1]{fontenc}
\usepackage{dsfont}
\usepackage{enumerate}
\usepackage[shortlabels]{enumitem}
\usepackage{fullpage}   
\usepackage{xcolor}
\usepackage[colorlinks=true,
linkcolor=blue,
citecolor=blue,
urlcolor=blue,
pagebackref,
]{hyperref}

\mathtoolsset{showonlyrefs}

\newcommand{\R}{\mathbb{R}}
\newcommand{\Z}{\mathbb{Z}}
\newcommand{\N}{\mathbb{N}}

\newcommand{\p}{\textrm{p}}
\newcommand{\vc}[1]{\boldsymbol{#1}}
\newcommand{\ind}{\mathds{1}} 	
\newcommand{\Sz}{\mathcal{S}} 	
\newcommand{\Mod}[1]{\;(\textrm{mod } #1)}

\DeclarePairedDelimiter\abs{\lvert}{\rvert}
\DeclarePairedDelimiter\norm{\lVert}{\rVert}

\DeclarePairedDelimiter\floor{\lfloor}{\rfloor}

\DeclareMathOperator{\supp}{supp}
\DeclareMathOperator{\BigO}{\mathcal{O}}

\DeclareFontFamily{U}{mathx}{\hyphenchar\font45}
\DeclareFontShape{U}{mathx}{m}{n}{
      <5> <6> <7> <8> <9> <10>
      <10.95> <12> <14.4> <17.28> <20.74> <24.88>
      mathx10
      }{}
\DeclareSymbolFont{mathx}{U}{mathx}{m}{n}
\DeclareFontSubstitution{U}{mathx}{m}{n}
\DeclareMathAccent{\widecheck}{0}{mathx}{"71}
\DeclareMathAccent{\wideparen}{0}{mathx}{"75}

\newtheorem{thm}{Theorem}[section]

\newtheorem{prop}[thm]{Proposition}
\newtheorem{lem}[thm]{Lemma}

\theoremstyle{definition}
\newtheorem{defin}[thm]{Definition}

\allowdisplaybreaks

\title{The Frisch--Parisi formalism for fluctuations of the Schr\"odinger equation}

\author{Sandeep Kumar, Felipe Ponce-Vanegas, Luz Roncal, and Luis Vega}

\address[S. Kumar]{
	Indominus Advanced Solutions S.L.\\ 36414 Vigo, Spain
}
\email{sandeepkumar.sssu@gmail.com}

\address[F. Ponce-Vanegas]{
BCAM -- Basque Center for Applied Mathematics\\
48009 Bilbao, Spain}
\email{fponce@bcamath.org}
 
\address[L. Roncal]{
BCAM -- Basque Center for Applied Mathematics\\
48009 Bilbao, Spain, Ikerbasque, Basque Foundation for Science, 48011 Bilbao, Spain, and Department of Mathematics\\
UPV/EHU\\
Apto. 644, 48080 Bilbao, Spain}
\email{lroncal@bcamath.org}

\address[L. Vega]{Department of Mathematics\\
 UPV/EHU\\
Apto. 644, 48080 Bilbao, Spain\\
and BCAM -- Basque Center for Applied Mathematics \\
48009 Bilbao, Spain}
\email{luis.vega@ehu.es}


\date{\today}

\dedicatory{Dedicated to Prof. Tohru Ozawa on the
occasion of his 60th birthday}

\begin{document}

\begin{abstract}
We consider the solution of the Schr\"odinger equation $u$ in
$\R$ when the initial datum tends to the Dirac comb.
Let $h_{\p, \delta}(t)$ be the fluctuations in time of 
$\int\abs{x}^{2\delta}\abs{u(x,t)}^2\,dx$, for $0 < \delta < 1$,
after removing a smooth background. 
We prove that the Frisch--Parisi formalism holds for
$H_\delta(t) = \int_{[0,t]}h_{\p, \delta}(2s)\,ds$,
which is morally a simplification of the Riemann's 
non-differentiable curve $R$.
Our motivation is to understand the evolution of the
vortex filament equation of polygonal filaments, which
are related to $R$. 
\end{abstract}


\maketitle


\section{Introduction}

The binormal curvature flow, also known as the vortex filament equation,
\begin{equation}
\label{BF}
\chi_t=\chi_x\wedge\chi_{xx} ,
\end{equation}
is a model for the dynamics 
of vortex filaments in Euler equations.
The function $\chi(t, x)$ describes a family of curves in 3d that move with time $t$
and are parametrized by arclength $x$.
Using the Frenet equations one easily concludes that the
right-hand side of \eqref{BF} is a vector whose modulus equals the curvature
and whose direction is the binormal vector.
By differentiating both sides by $x$,
we get the one dimensional Schr\"odinger map
\begin{equation}
\label{SM}
T_t=T\wedge T_{xx}, \qquad \text{where } T := \chi_x \in \mathbb{S}^2.
\end{equation}
Our interest in this paper is in curves that can develop corners in finite time.
For that purpose, it is better to use the so-called parallel frame $(T,e_1,e_2)$ instead of
the usual Frenet frame, where the former is defined by
\begin{equation}
\label{Tx}
\begin{pmatrix}
	T_x \\ (e_1)_x \\ (e_2)_x
\end{pmatrix}
=
\begin{pmatrix}
	0 & \alpha & \beta \\
	-\alpha & 0 & 0 \\
	-\beta & 0 & 0
\end{pmatrix}\cdot
\begin{pmatrix}
	T \\ e_1 \\ e_2
\end{pmatrix}.
\end{equation}
Hasimoto proved in \cite{Ha} that for $T$ to be a solution of \eqref{SM},
$u := \alpha + i\beta$ has to solve the 1d cubic non-linear Schr\"odinger equation
\begin{equation}
\label{NLS}
iu_t+ u_{xx}+\frac12((|u|^2-A(t))u=0,
\end{equation}
for some  real function $A(t)$;
Hasimoto used the Frenet frame, but
the proof admits more general frames.

In \cite{hozVega2014}, de la Hoz and the fourth author studied the evolution
of regular planar polygons $\chi_M$, with $M$ denoting the number of sides.
In particular, they were interested in the trajectories described by any of the corners.
That is to say, if we assume that  at time $t=0$ there is a corner at the origin,
then they studied the curve in 3d 
\begin{equation}
\label{RM}R_M(t):=\chi_M (t,0).
\end{equation}
These curves show a characteristic fractal behavior which is reminiscent
of the so-called Riemann's non-differentiable function.
In fact, they
found compelling numerical evidence that $\lim_{M \to\infty}R_M(t)=R(t)$ with
\begin{equation}
R(t) := \int_0^t u_D(0, s)\,ds,
\end{equation} 
where $u_D$ is the solution of the linear Sch\"odinger equation
with initial datum $F_D = \sum_{n \in \Z}\delta_n$, that is,
$F_D$ is the Dirac's comb.  It turns out that $R$ is a small modification
of the complex version of Riemann's function
\begin{equation}
	\phi(t) := \sum_{n=1}^\infty \frac{e^{i\pi n^2t}}{i\pi n^2}
	= 2\pi iR\Big(\frac{-t}{4\pi}\Big) - \frac{t}{2} - \frac{i\pi}{6}.
\end{equation}
We notice that $u_D$ exhibits the Talbot effect, that is,
the appearence of rescaled and weighted Dirac combs
at rational times, which easily justifies the fractal appearence of $R$. 
There is a rich literature about the Talbot effect;
see for example \cite{berry96, erdoganBook, rodnianski98, taylor03, oskolkov10}. 

Recently, Banica and the fourth author \cite{BV} tightened the connection 
between $R$ and the binormal curvature flow. 
They proved that the evolution of a corner of a suitably chosen sequence of polygonal vortex filaments approaches $R(t)$
in the limit when the number of sides is infinite. 
Additionally, inspired by the work of Jaffard \cite{jaffard96}, 
they showed that the limiting behavior of the corners falls 
within the multifractal formalism of Frisch and Parisi, 
which is conjectured to govern turbulent fluids.
By analogy with turbulence,
we would expect that $R$ could be understood as the
outcome of some stochastic process;
such an interpretation still seems to be missing.

Yet another interesting physical phenomenon, 
which is closely related to multifractality, 
is the intermittency. 
Roughly speaking, the idea is that the velocity of a fluid in fully developed turbulence 
may erratically change in very small distances, 
suggesting a very irregular structure. 
This phenomenon, called \textit{intermittency in small scales} is related to the Frisch--Parisi multifractal formalism, 
but it does not seem to be well-defined in the literature. 
In \cite{BVE}, by adapting the physical concept of intermittency to the setting of functions and 
giving a precise definition, 
the authors gave quantitative estimates of the intermittency of 
the Riemann's non-differentiable function.

Within this circle of ideas, in \cite{kumar2021},
during an investigation of the dispersive
properties of the free Schr\"odinger equation,
an interesting behavior 
was discovered for the functional
\begin{equation}
h_\delta[f](t) := \int\abs{x}^{2\delta}\abs{u(x,t)}^2\,dx,\qquad x\in \R^d
\qquad \textrm{for } 0 < \delta < 1,
\end{equation}
where $u$ is the solution of the linear Schr\"odinger equation
with initial datum $f$.
By renormalization (removing an infinite and rescaling), 
the authors extended the definition of $h_\delta$
to periodic initial data like the Dirac comb $F_D$;
let us call $h_{\textrm{p}, \delta}[f]$ (p for periodic) to
the renormalization. 

During the renormalization of $h_\delta[f]$ 
a smooth function is removed,
leaving behind small fluctuations that  
approach the point function in Figure~\ref{fig:h_p_Dirac} when
$f$ approaches the Dirac comb $F_D$.
The function $h_{\textrm{p}, \delta}[F_D]$ is
supported at rationals,
so it is somehow a simplification of $u_D$,
which has a complex structure at irrational times.
This simplification offers the possibility of understanding hard questions
associated with $u_D$ by considering first $h_{\textrm{p},\delta}[F_D]$.

\begin{figure}[t] 
\includegraphics[width=0.52\textwidth]{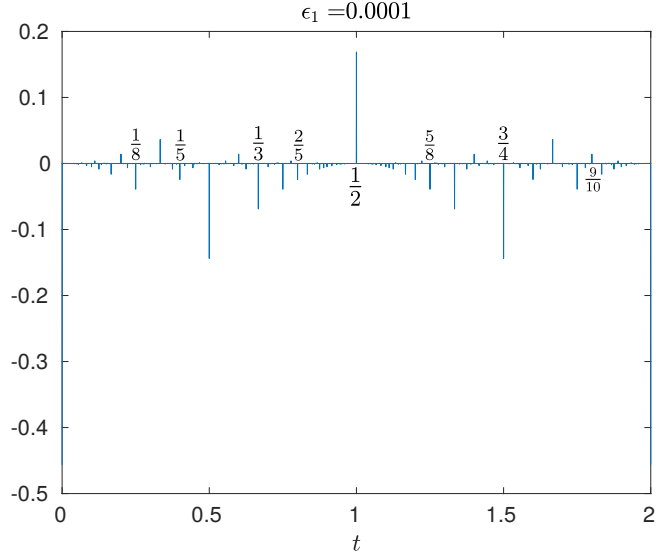}
\caption{Evolution of $h_{\textrm{p},\delta}[f_{\epsilon}]$, 
where $f_{\epsilon}$ is a smooth periodic
function that approaches the Dirac comb
in the sense of distributions as
$\epsilon \to 0^+$.}
\label{fig:h_p_Dirac}
\end{figure}

In \cite{kumar2021}
the authors exposed evidences showing that
\begin{equation}
\label{eq:Hdelta}
H_\delta(t) = 
	\int_{[0,t]} h_{\textrm{p},\delta}[F_D](s)\,ds,
\qquad \textrm{for } 0 < \delta < 1,
\end{equation}
can be seen as the outcome of a $(2/s)$-L\'evy process
with $s := 2(1 + \delta)$.
Unbeknownst to the authors,
similarities between the Riemann non-differentiable
function (which is behind $u_D$) 
and L\'evy processes had already been pointed out
by Jaffard in \cite[Sections~2.3 and 4.4]{fractalsEng97}.

The velocity of turbulent flows differs widely from
point to point, so in this context it has been
introduced the spectrum of singularities of a function, which measures the size of the sets
with different H\"older exponents.

\begin{defin}[H\"older exponent] \label{def:Holder_exp}
Let $f$ be a function and $t\in\R$. 
A function $f \in C^l(t)$, for real $l\ge 0$,
if there is a polynomial $P_t$ of degree at most 
$\floor{l}$ such that in a neighborhood of $t$
\begin{equation}
\abs{f(s) - P_t(s)} \lesssim \abs{t - s}^l.
\end{equation}
The H\"older exponent of $f$ at $t$ is
\begin{equation}
h_f(t) := \sup\{l \mid f \in C^l(t)\}.
\end{equation}
\end{defin}

To measure the size of a set, 
we use the concept of Hausdorff dimension.

\begin{defin}[Hausdorff dimension] \label{def:Hausdorff_Dim}
Let $A\subset \R^n$ and $R_{\varepsilon}$ be the set of all coverings of $A$ by sets $A_i$ of diameter at most $\varepsilon$. Let 
\begin{equation}
\mathcal{H}_\varepsilon^d(A)
	:=\inf_{r\in R_{\varepsilon}}\sum_{A_i\in r}(\operatorname{diam}A_i)^d.
\end{equation}
Then,
\begin{equation}
\mathcal{H}^d(A)
	:= \lim_{\varepsilon \to 0} \mathcal{H}_\varepsilon^d(A)
\end{equation}
is the $d$-dimensional Hausdorff content of $A$.
The \textit{Hausdorff dimension} of $A$ is
\begin{equation}
\dim_{\mathcal{H}} A :=\inf\{d:\mathcal{H}^d(A)=0\}
	=\sup\{d:\mathcal{H}^d(A)=+\infty\}.
\end{equation}
\end{defin}

Now we can define the spectrum of singularities of a function.

\begin{defin}[Spectrum of singularities] \label{def:SingSpec}
Let $f$ be a function and define the set
\begin{equation}
\Gamma_h := \{t \in \R \mid f\textrm{
 has H\"older exponent } h \textrm{ at } t\}.
\end{equation}
The spectrum of singularities is the function
\begin{equation}
D_f(h) = \dim_{\mathcal{H}} \Gamma_h.
\end{equation}
If $\Gamma_h = \emptyset$, then $D_f(h) = -\infty$.
\end{defin}

As we mentioned, in \cite[Theorem~1(iii)]{BV} it was proved
that the spectrum of singularities of $R$ (and modifications of it) is
\begin{equation} \label{eq:SingSpectrum_R}
	D_{R}(h) = 4h - 2, \qquad 
		\text{for all } h \in \Big[\frac{1}{2}, \frac{3}{4}\Big].
\end{equation}
Concerning $H_\delta$, 
it was proved in \cite[Theorem 4]{kumar2021} that
the spectrum of singularities of $H_\delta$ is
\begin{equation} \label{eq:spectrum_sing_H}
D_{H_\delta}(h) = \begin{cases}
\alpha h, & \textrm{if } h \in [0, 1/\alpha], \\
-\infty, & \textrm{if } h > 1/\alpha,
\end{cases}
\end{equation}
where $\alpha = 2/s$ and $s = 2(1 + \delta)$, for $0<\delta<1$.
Surprisingly, for $\alpha$-L\'evy processes 
Jaffard proved in \cite{jaffard99} that the 
spectrum of singularities is almost surely equal to \eqref{eq:spectrum_sing_H}.
Before stating our main result, 
we describe briefly the Frisch--Parisi formalism.

\subsection{Frisch--Parisi formalism}

The so-called multifactral formalism 
for functions relates some functional norms of a function 
to its spectrum of singularities. 
This formalism was introduced by Frisch and Parisi in order to numerically 
determine the spectrum of fully turbulent fluids \cite{FP}. 
Even though the Frisch-Parisi formalism has
several versions,
we decided to use the \textit{wavelet--transform integral
method} as described at the introduction of \cite{jaffardI97}.
First we must define the wavelet transform of a function.

\begin{defin}
\label{def:wt}
The \textit{wavelet transform} of a function $f$ is defined as $\psi_N\ast f$, 
where $\psi_N(x) := N\psi(Nx)$. 
The wavelet $\psi$ is a function whose smoothness
and decay are adjusted depending on 
the problem, and such that
\begin{equation}
\int x^k\psi(x)\,dx = 0, \qquad
	\textrm{for } k = 0, \ldots, L
	\enspace
	\textrm{and some suitable } L.
\end{equation}
\end{defin}

The Frisch--Parisi formalism suggests that
the spectrum of singularities can be computed through
the \textit{scaling exponent} $\eta_f$, 
which is defined by 
\begin{equation} \label{eq:def:scaling_exp}
\eta_f(p) := -\liminf_{N\to \infty}
	\frac{\log\, \norm{\psi_N\ast f}^p_{L^p}}{\log N}.
\end{equation}
The Legendre transform provides a link
between $D_f$ and $\eta_f$ through the conjectured relationship: 
\begin{equation}
D_f(h) = \inf_{p > 0}(ph - \eta_f(p) + 1).
\end{equation}

\subsection{Main result}

Since $H_\delta[F_D]$ might be seen as a 
simplification of the Riemann's non-di\-ffe\-ren\-tia\-ble
function, for which the Frisch-Parisi formalism
has been proved (see \cite{jaffard96} or \cite{BV}), 
then we should be able to prove the
Frisch--Parisi formalism for $H_\delta[F_D]$ in
the range $[0, s/2]$.
We confirm this in our main theorem below.

\newpage 

\begin{thm}[Frisch--Parisi formalism]
\label{thm:FP}
Let $0<\delta<1$ and $s := 2(1+\delta)$.
Let $\psi$ be an integrable function such that:
\begin{enumerate}[(i)]
\item $\int_\R \psi = 0$.
\item $\abs{\psi(x)} \lesssim x^{-\beta}$, for $\beta > 1 + s$.
\end{enumerate}
Let $\eta_{H_\delta}$ be the scaling exponent defined in \eqref{eq:def:scaling_exp}
for the function $H_\delta$.
Then,
\begin{equation}
\eta_{H_\delta}(p) = \begin{cases}
sp/2, & \textrm{if } 0 < p \le 2/s, \\
1, & \textrm{if } p \ge 2/s.
\end{cases}
\end{equation}
In particular, 
the Frisch--Parisi formalism holds 
in the range $[0, s/2]$.
\end{thm}

\begin{center}
\includegraphics[width=0.45\textwidth]{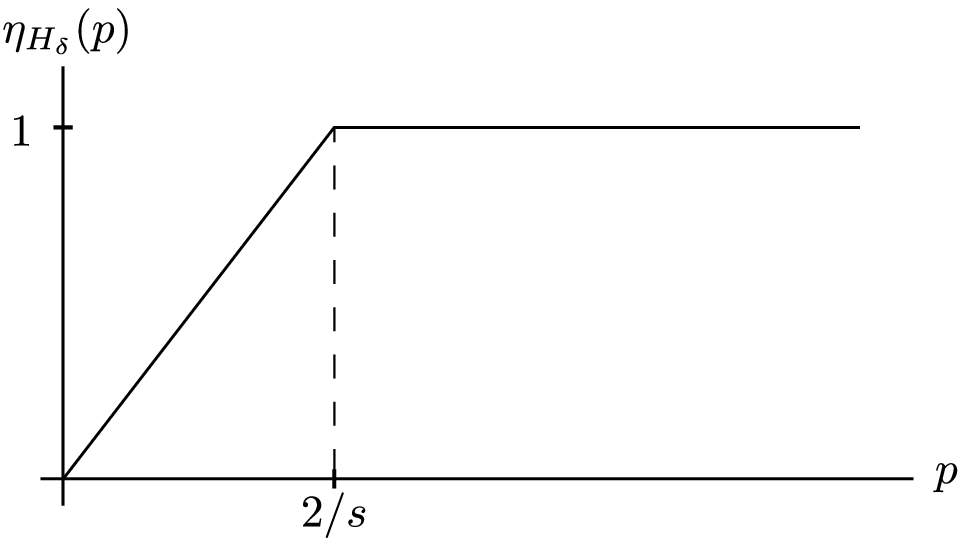}
\end{center}

As already said, we see $H_\delta[F_D]$ as a 
simplification of the Riemann's non-differentiable function $R$.
Therefore, a natural question is to determine
whether the multifractal formalism holds true for the non-linear trajectories $R_M$ given in \eqref{RM}.
This seems to be a very challenging question at the theoretical level,
so, to gain insight into the subject,
we computed numerically the spectrum of singularities of $R_M$
for several values of $M$; see Figure~\ref{fig:HdeltaRNDF}.A.
As a matter of comparision,
we also computed numerically the spectrum
of $R$ (Fig. \ref{fig:HdeltaRNDF}.A) and $H_\delta$ (Fig. \ref{fig:HdeltaRNDF}.B),
for which the theoretical values are known.
In Appendix~\ref{app:numerical} we describe the methods
used to compute the spectrum of singularities.
Although more careful experiments are needed,
they suggest that the spectrum of singularities of $R_M$
should be equal to that of $R$.

We wonder whether it is possible to define $H_\delta$
replacing $u_D$ by the solution of the non-linear Schr\"odinger equation, and in that case,
whether the resulting $H_\delta$ and its spectrum of singularities
is more amenable to theoretical studies.

\begin{figure}[t]
	\includegraphics[width=0.96\textwidth]{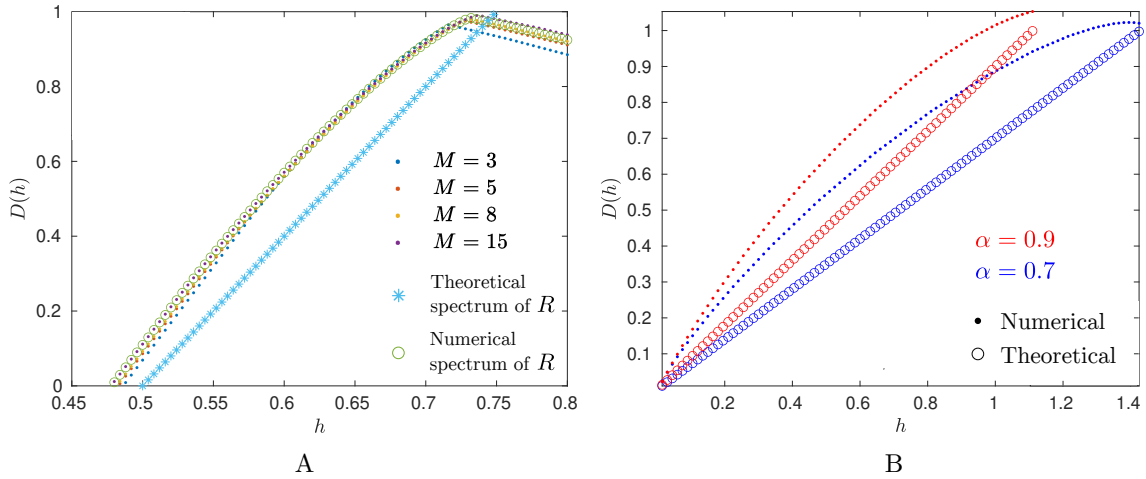}
	\caption{The spectrum of singularities $D(h)$ estimated
		using the \textit{wavelet--transform modulus maxima}  method for: 
		A) $M$-sided polygons $R_M$ with different $M$ values and Riemann's function $R$, and
		B) $H_\delta$ with two values of $\alpha := 1/(1 + \delta)$.  
		Clearly, up to the numerical errors, it captures the support of $D(h)$ very well in both cases.
		See Appendix~\ref{app:numerical} for more details about the numerical methods.}
	\label{fig:HdeltaRNDF}
\end{figure}

\subsection*{Notation}

We write $A \lesssim B$ if $A \le CB$ for some constant $C > 0$;
the relations $\gtrsim$ and $\simeq$ are similar.
We also write $\norm{f}_p = \norm{f}_{L^p([0,1])}$.

The H\"older exponent is given in Definition~\ref{def:Holder_exp};
the Hausdorff dimension is Defintion~\ref{def:Hausdorff_Dim};
the spectrum of singularities is Definition~\ref{def:SingSpec}; and
the scaling exponent is \eqref{eq:def:scaling_exp}.

\subsection*{Funding}

This work was supported by the Basque Government 
(BERC 2022-2025 program) and by the Spanish State Research Agency
(Severo Ochoa SEV-2017-0718).
The second author was funded by
the project PGC2018-094528-B-I00 - IHAIP and by
a Juan de la Cierva--Formation grant FJC2019-039804-I.  
The third author acknowledges the project PID2020-113156GB-I00, 
the RyC project RYC2018-025477-I, and Ikerbasque. 
The fourth author is supported by ERCEA Advanced Grant 2014 669689 - HADE and 
the project PGC2018-094522-B-I00. 


\section{Proof of Theorem~\ref{thm:FP}}

To prove Theorem~\ref{thm:FP} 
it is more convenient to work with $h_{\p, \delta}$
rather than directly with $H_\delta$.
Since $\psi$ has vanishing mean, 
we can write it as $\psi = \phi'$.
Thus, the operator $\psi_N\ast H_\delta$ is expressed as
\begin{equation}
\label{eq:psiHd}
\int \psi_N(x - y) H_\delta(y) \,dy =
	-\int \frac{d}{dy}\phi(N(x - y))\, H_\delta(y)\,dy = \\
	\int \phi(N(x - y))h_{\textrm{p},\delta}(y) \,dy.
\end{equation}
Here, the wavelet $\phi$ has the following properties: for some $c_1,c_2>0$,
\begin{itemize}
\item $|\phi(x)|\lesssim 1$, for $|x|\le c_1$,
	because $\psi$ is integrable;
\item decay of the tails: for some $\alpha > s$,
\begin{equation}
\label{eq:decay}
|x|^\alpha\abs{\phi(x)} \le c_2, \qquad
	\textrm{for } \abs{x} \ge c_1,
\end{equation}
because of the decay of $\psi$;
\item the $L^p$-norm is concentrated
\begin{equation}
\int_{\abs{x}\le c_1}\abs{\phi}^p \,dx \ge \frac{1}{2}\int\abs{\phi}^p \,dx.
\end{equation} 
\end{itemize}
Along the paper, we will be using systematically the properties of $\phi$ without further comment.

Now let us write out 
the distribution $h_{\textrm{p}, \delta}[F_D] \in \Sz'(\mathbb{T})$:
\begin{equation}
\label{eq:hpdelta}
h_{\textrm{p}, \delta}[F_D](x) = 
	\sum_{(p,q) = 1}\frac{a_{q,\delta}}{q^s}\delta_{p/q}(x),
\end{equation}
where $s = 2(1 + \delta)$ and
\begin{equation}
a_{q,\delta} = 
	-2b_{1,\delta}\zeta(2(1+\delta))
	\begin{cases}
		1, & \textrm{if } q \textrm{ is odd,} \\
		-2(2^{1 + 2\delta} - 1), & \textrm{if }
			q \equiv 2\Mod{4}, \\
		2^{2(1 + \delta)}, & \textrm{if }
			q \equiv 0\Mod{4}.
	\end{cases} 
\end{equation}
Here $\zeta$ is the Riemann zeta function and
\begin{equation}
b_{1,\delta} = 
	\frac{1}{(2\pi)^{2\delta}}\frac{\Gamma(2\delta)}{\abs{\Gamma(-\delta)}\Gamma(\delta)}.
\end{equation}
At the end,
the only property of $a_{q,\delta}$ we will make use of
is $\abs{a_{q,\delta}} \simeq_\delta 1$. 

Let us denote the last integral in \eqref{eq:psiHd} by $P_N h_{\textrm{p},\delta}$, so in view of \eqref{eq:hpdelta}
\begin{equation}
P_N h_{\textrm{p},\delta}(x) = 
	\sum_{p/q}
		\frac{a_{q,\delta}}{q^s}\phi(N(x - p/q)),
\end{equation}
where from here onwards, 
the notation $\sum_{p/q}$ stands for the sum over all pairs of integers $p,q$ such that $(p,q) = 1$. 
Without loss of generality, we will assume that $q$ is nonnegative. 
There should not be confusion between the appearance of $p$ as an integer and in the $L^p$ norms.
We aim to prove the next theorem,
from which our main theorems follow.

\begin{thm} \label{thm:Lp_norm_N}
Let $0< p\le \infty$, $0<\delta<1$ and $s=2(1+\delta)$. Then, for $N\gg1$,
\begin{equation}
\norm{P_N h_{\textrm{p},\delta}}_{L^p([0,1])} \simeq_{\delta}
\begin{cases} 
N^{-1/p}, &\textrm{if } 2/s\le p \le \infty, \\
N^{-s/2}(\log N)^{s/2}, &\textrm{if } p = 2/s, \\
N^{-s/2}, &\textrm{if } 0<p \le 2/s.
\end{cases}
\end{equation}
\end{thm}

To estimate the $L^p$ norm of $P_Nh_{\p,\delta}$, 
the idea is to split this function as
\begin{align}
P_N h_{\textrm{p},\delta}(x) &= 
	\sum_{\substack{p/q \\ q \le c_0\sqrt{N}}}
		\frac{a_{q,\delta}}{q^s}\phi(N(x - p/q)) \; +
		&& (\textrm{main term}\enspace \vc{M}(x)) \\
	&\sum_{\substack{p/q \\ q > c_0\sqrt{N}}}
		\frac{a_{q,\delta}}{q^s}\phi(N(x - p/q)). 
		&& (\textrm{error term}\enspace \vc{E}(x))
\end{align}
Here, $0 < c_0 \ll 1$ is a constant to be fixed later.
In light of the inequality
\begin{equation}
\abs{\norm{P_N h_{\p,\delta}}_p - \norm{\vc{M}}_p} \le 
	\norm{\vc{E}}_p,
\end{equation}
the goal is to get an estimate for $\norm{\vc{M}}_p$
and a suitable upper bound of $\norm{\vc{E}}_p$.

\subsection{The role of $\vc{M}$}

This subsection is philosophical in nature.
We want to discuss what is the relationship
between $\vc{M}$ and the spectrum of singularities
of $H_\delta$.

In \cite[Theorem 4]{kumar2021}, 
during the proof of \eqref{eq:spectrum_sing_H},
it is actually shown that
the H\"older exponent of $H_\delta$ 
is $s/\mu$, for $\mu > 2$,
exactly in the set of numbers $\Gamma_\mu$
with irrationality $\mu$.
Let us recall the definition of irrationality.

\begin{defin}[Irrationality Measure]
A number $x$ has irrationality $\mu$ if
for every $\eta < \mu$ 
there are infinitely many rationals $p/q$ such that
\begin{equation}
0 < \Big|x - \frac{p}{q}\Big| < \frac{1}{q^\eta},
\end{equation}
but for $\eta > \mu$ there are at most finitely many.
\end{defin} 

The set $\Gamma_\mu$ of numbers with irrationality $\mu$ 
has Hausdorff dimension $2/\mu$, which is
consequence of Jarn\'ik's \cite[Theorem 1]{Jarnik1931}, that is,
\begin{gather}
\dim W = 2/\mu \qquad\mbox{and}\qquad 
	\mathcal{H}^{2/\mu}(W) = +\infty, \\
\text{ where } \quad W = \Big\{x \mid \Big|x - \frac{p}{q}\Big| < \frac{1}{q^\mu} 
\textrm{ for infinitely many rationals } p/q\Big\}.
\end{gather} 

Now if we consider $\vc{M}$, 
it is essentially supported around fractions
$p/q$ with $q \le \sqrt{N}$;
let us forget about the parameter $c_0$,
which is introduced for technical reasons.
We can decompose $\vc{M}$ in a dyadic parameter $\lambda$ as
\begin{equation}
\vc{M}(x) = \sum_{\lambda \le \sqrt{N}}
	\sum_{q \simeq \lambda} 
	\frac{a_{q,\delta}}{q^s}\phi(N(x - p/q))
=: \sum_{\lambda \le \sqrt{N}} \vc{M}_\lambda(x).
\end{equation}
Hence, ignoring the tail of $\phi$,
$\vc{M}_\lambda$ is supported in a set $V_\lambda$
which is union of $\simeq \lambda^2$
pairwise disjoint intervals of length $1/N$.
For each $x \in V_\lambda$ we can find a fraction $p/	q$ such that
\begin{equation}
\Big|x - \frac{p}{q}\Big| < \frac{1}{N},
\end{equation} 
where $q \simeq \lambda \simeq N^{1/\mu}$ for some $\mu \ge 2$.
We might see $V_\lambda$ as a blurring of $\Gamma_\mu$
at scale $1/N$, so
let us rename $V_\lambda$ as ``$\Gamma_\mu$'' and
notice that $\abs{``\Gamma_\mu"} \simeq N^{2/\mu - 1}$,
which is what we would expect of a blurring
at scale $1/N$ of a set of dimension $2/\mu$.

We can compute heuristically the H\"older dimension
of $P_N H_\delta$ in ``$\Gamma_\mu$'', where $H_{\delta}$ was defined in \eqref{eq:Hdelta}.
Neglecting $\vc{E}$ (we are being overbold here), 
for $x,y\in ``\Gamma_\mu"$ with $\abs{x - y}\simeq 1/N$ we would have
\begin{equation}
\abs{H_\delta(x) - H_\delta(y)} \simeq
	\vc{M}(x) \simeq N^{-s/\mu}.
\end{equation}
Hence, in $``\Gamma_\mu"$ the H\"older exponent would be $s/\mu$, 
which agrees heuristically with  \cite[Theorem 4]{kumar2021}.
The Frisch--Parisi formalism is thus reflected 
in the $L^p$ norm of $\vc{M}$,
see for instance \eqref{eq:Lp_norm_M_above}.
The strength of the relationship between $\vc{M}$ and $D_{H_\delta}$
fades away as $\mu \to 2$,
and at $\mu = 2$ the ``error term'' $\vc{E}$
takes the main role.

\subsection{Proof of Theorem \ref{thm:Lp_norm_N}: the upper bound}
\label{sub:upper}

First, we bound $P_N h_{\textrm{p},\delta}$ pointwise with 
a simpler function.
Since $\phi$ decays strongly, 
then to control it pointwise
we can tile $\R$ with intervals $J$ with some
suitable length $\abs{J} = c_1$ so that
\begin{equation}
\phi(x) \le \sum_J b_J\ind_{J}(x),
\end{equation}
where $b_J$ are coefficients decaying very fast; 
the tiling is so that
one of the intervals, call it $J_0$, is centered at the origin.
Hence, we can write
\begin{equation}
\norm{P_N h_{\textrm{p},\delta}(\cdot)}_p \lesssim
	\sum_J b_J \Big\lVert\sum_{p/q}\frac{1}{q^s}
		\ind_J(N(\cdot - p/q))\Big\lVert_p.
\end{equation}
By translation symmetry, it suffices to consider $J = J_0$.

Let $\{K\}$ be a tiling of $\R$ with intervals
of length $4c_1/N$, one of them centered at zero,
and let $\{K'\}$ be another tiling equal to $\{K\}$
but shifted by $2c_1/N$.
Since every interval $I_{p/q}$ of length $2c_1/N$ and
centered at $p/q$ is
contained in one interval either from $\{K\}$ or from $\{K'\}$, then
\begin{equation}
\sum_{p/q}
	\frac{1}{q^s}\ind_{I_{p/q}}(x) \le 
\sum_K\sum_{p/q \,\in\, K}
	\frac{1}{q^s}\ind_K(x) +
\sum_{K'}\sum_{p/q \,\in\, K'}
	\frac{1}{q^s}\ind_{K'}(x).
\end{equation}
Let us assume that $c_0 < (2c_1)^{-1/2}$. Hence, it suffices to control the $L^p$ norm of
\begin{equation}
\sum_K\sum_{p/q \,\in\, K}
	\frac{1}{q^s}\ind_K = 
\sum_K\sum_{\substack{p/q \,\in\, K \\ q\le c_0\sqrt{N}}}
	\frac{1}{q^s}\ind_K + 
\sum_K\sum_{\substack{p/q \,\in\, K \\ q > c_0\sqrt{N}}}\frac{1}{q^s}\ind_K 
=: \vc{M} + \vc{E}.
\end{equation}

Now we prove an upper bound for $\vc{E}$.

\begin{lem} \label{thm:Error_U}
Let $0 < p\le \infty$. 
Let $0<\delta<1$ and $s=2(1+\delta)$.
Then, for $N\gg1$,
\begin{equation}
\|\vc{E}\|_{p}\lesssim_{\delta} 
\begin{cases}
N^{-\frac{s}{2}+\frac{1}{2p'}}, & 
	\textrm{if } 1 \le p \le \infty \\
N^{-\frac{s}{2}}, &
	\textrm{if } 0 < p < 1.
\end{cases}
\end{equation}
\end{lem}
\begin{proof}
We begin with the range $1 \le p \le \infty$.
We first estimate the $L^1$ norm. 
Let $\varphi$ be the Euler's totient function\footnote{The Euler's totient function
$\varphi(n)$ is the number of integers $k$, $1\le k\le n$, 
such that $(n,k)=1$}, then
\begin{align}
\|\vc{E}\|_{1}&\lesssim_{\delta}
\sum_K\sum_{\substack{p/q \,\in\, K \\ q > c_0\sqrt{N}}}
	\frac{1}{q^s}\abs{K} \\
	&\lesssim_{\delta} \frac{1}{N}\sum_{q > c_0\sqrt{N}}
		\frac{\varphi(q)}{q^{s}}\\
&\lesssim_{\delta} \frac{1}{N}\sum_{q > c_0\sqrt{N}}
		\frac{1}{q^{s-1}}
	\qquad (\textrm{by } \varphi(q) \le q)\\
&\lesssim_{\delta}\frac{1}{N}(\sqrt N)^{-s+2}=N^{-s/2}.
\end{align}
On the other hand,
\begin{equation}
\|\vc{E}\|_{\infty} \lesssim_{\delta}
	\sup_{|K|=2c_1/N}	
	\sum_{\substack{p/q \,\in\, K \\ q > c_0\sqrt{N}}}\frac{1}{q^{s}}.
\end{equation}
Fix an interval $K$.
If $kN/(2c_1) \le q < (k + 1)N/(2c_1)$, 
then there are at most $k+1$ rationals $p/q \in K$, so
\begin{align}
\sum_{\substack{p/q \,\in\, K \\ q > c_0\sqrt{N}}}\frac{1}{q^{s}} &\le
\sum_{c_0\sqrt{N} < q \le N/(2c_1)}\frac{1}{q^{s}} + 
\sum_{k \ge 1}\sum_{k \le 2c_1q/N < k+1}\frac{k+1}{q^s} \\
&\lesssim N^{(1 - s)/2} + N^{1 - s} \\
&\lesssim N^{(1 - s)/2}.
\end{align}
Now, interpolate the two above estimates: for $\theta=1/p$, $p>1$,
\begin{align}
\|\vc{E}\|_{p}&\lesssim \|\vc{E}\|_{1}^{\theta}\|\vc{E}\|_{\infty}^{1-\theta}\\
&\lesssim (N^{-s/2})^{\theta}(N^{\frac{1-s}{2}})^{1-\theta}\\
&=N^{-\frac{s}{2}+\frac{1-\theta}{2}}=N^{-\frac{s}{2}+\frac{1}{2p'}},
\end{align}
as desired.

For the range $0 < p < 1$  we use H\"older with $r = 1/p$ so that
$\norm{\vc{E}}_p=\norm{\vc{E}}_{L^p([0,1])} \le \norm{\vc{E}}_1 \lesssim N^{-s/2}$.
\end{proof}

It remains to bound the main term $\vc{M}$.
Notice that each interval $K$ with $|K|=\frac{2c_1}{N}$ contains at most 
one rational $p/q$ with $q \le c_0\sqrt{N}$ whenever 
\begin{equation}
\label{eq:condc0c1}
c_0 < (2c_1)^{-1/2}. 
\end{equation}
Indeed, this follows by contradiction, since $\big|\frac{p_1}{q_1}-\frac{p_2}{q_2}\big|\ge \frac{1}{q_1q_2}\ge \frac{1}{c_0^2N}$, assuming $q_1,q_2\le c_0\sqrt{N}$.
Hence, we can control the $L^p$ norm as
\begin{equation}
\norm{\vc{M}}_p^p = 
	\sum_K\sum_{\substack{p/q \,\in\, K \\ q\le c_0\sqrt{N}}}
	\frac{1}{q^{ps}}\abs{K} \lesssim 
\frac{1}{N}\sum_{\substack{p/q \\ q\le c_0\sqrt{N}}}
	\frac{1}{q^{ps}}\lesssim  
\frac{1}{N}\sum_{q\le c_0\sqrt{N}}\frac{1}{q^{ps-1}},
\end{equation}
where we used $\varphi(q) \le q$.
We compute the last sum to get
\begin{equation} \label{eq:Lp_norm_M_above}
\norm{\vc{M}}_p \lesssim
\begin{cases}
N^{-1/p}, &
	\textrm{if } p > 2/s, \\
N^{-s/2}(\log N)^{s/2}, &
	\textrm{if } p = 2/s, \\
N^{-s/2}, &
	\textrm{if } p < 2/s.
\end{cases}
\end{equation}
Inequality \eqref{eq:Lp_norm_M_above} and Lemma~\ref{thm:Error_U} imply the upper bound in 
Theorem~\ref{thm:Lp_norm_N}. 

\subsection{Proof of Theorem \ref{thm:Lp_norm_N}: the lower bound}

Before getting to the proof of the lower bound, 
we need a quite technical lemma that
says that the tail of $\phi$ can be safely ignored. 

Recall that for $q \le c_0\sqrt{N}$ with $c_0 \ll 1$, 
we denote by $I_{p/q}$ the interval of length $2c_1/N$
centered at $p/q$;
we assume that $c_0 < (2c_1)^{-1/2}$ to ensure that
the intervals are disjoint.
We pick a rational $p_0/q_0$ and 
define the error function
\begin{equation}
e(x) = 
	\sum_{p/q\, \notin \,2I_{p_0/q_0}}
	\frac{a_{q,\delta}}{q^s}\phi(N(x - p/q)),
\qquad x \in I_{p_0/q_0}.
\end{equation}
We will show that $e$,
which is the sum of the tails in $I_{p_0/q_0}$,
is small. 
The reader can skip the next lemma 
under the assumption $\supp \phi\subset [-c_1, c_1]$.

\begin{lem}[Tails are negligible]
	\label{thm:Tails_Negligible}
Let $s>2$ and $q_0 \le c_0\sqrt{N}$ for $c_0 \ll 1$. Then,
\begin{equation}
\abs{e(x)} \lesssim 
\frac{q_0^{s- 2}}{N^{s - 1}},
\qquad \textrm{for } x \in I_{p_0/q_0}.
\end{equation}
\end{lem} 

\begin{proof}
Since $\abs{x - p/q} \ge c_1/N$ then
\begin{equation} \label{eq:Upper_e_1}
\abs{e(x)} \le C\frac{c_2}{N^\alpha}
	\sum_{\substack{p/q \\ p/q\, \notin \,2I_{p_0/q_0}}}
	\frac{1}{q^s}\frac{1}{\abs{x - p/q}^\alpha}.
\end{equation}
Moreover, we can write $x = p_0/q_0 + \delta x$, with $\abs{\delta x} \le c_1/N$, then
\begin{equation}
\Big|x-\frac{p}{q}\Big| \ge 
	\Big|\frac{p_0}{q_0}-\frac{p}{q}\Big|
		- \frac{c_1}{N}
\ge \frac{1}{2}\Big|\frac{p_0}{q_0}-\frac{p}{q}\Big|.
\end{equation}
We replace the above in \eqref{eq:Upper_e_1} so that
\begin{equation}
\abs{e(x)} \le 
	C\frac{q_0^\alpha}{N^\alpha}
	\sum_{\substack{p/q \\ p/q\, \notin \,2I_{p_0/q_0}}}
		\frac{q^{\alpha - s}}{\abs{qp_0 - q_0p}^\alpha}.
\end{equation}
Since $p/q \notin 2I_{p_0/q_0}$ is the same as
\begin{equation}
\abs{qp_0 - q_0p} \ge 2c_1\frac{q_0q}{N},
\end{equation}
then it is sensible to break the sum above as
\begin{equation}
\abs{e(x)} \le C\frac{q_0^\alpha}{N^\alpha}
	\sum_{k \ge 0}\;\sum_{\{q\,:\,k < 2c_1q_0q/N \le k + 1\}}\;
	q^{\alpha - s}\sum_{\{p\,:\,\abs{qp_0 - q_0p} > k\}} 
	\frac{1}{\abs{qp_0 - q_0p}^\alpha}.
	\label{eq:Upper_e_2}
\end{equation}

To estimate the very last sum in \eqref{eq:Upper_e_2},
let $\overline{a}$ denote a residue of $a$ mod $q_0$
such that $\abs{\overline{a}} \le q_0/2$, so
if $\abs{\overline{qp}_0} > k$ then 
$\{l\,:\,\abs{\overline{qp}_0 + lq_0}> k\} =
\Z$ and
\begin{equation}
\sum_{\{p\,:\,\abs{qp_0 - q_0p}> k\}} \frac{1}{\abs{qp_0 - pq_0}^\alpha} = 
\sum_{\{l\,:\,\abs{\overline{qp}_0 + lq_0}> k\}}
	\frac{1}{\abs{\overline{qp}_0 + lq_0}^\alpha}
\lesssim \frac{1}{\abs{\overline{qp}_0}^\alpha}
\qquad (\textrm{assume } \alpha > 1).
\end{equation}
If $\abs{\overline{qp}_0} \le k < q_0/2$ then
\begin{equation}
\sum_{\{l\,:\,\abs{\overline{qp}_0 + lq_0}> k\}} \frac{1}{\abs{\overline{qp}_0 + lq_0}^\alpha}
\lesssim \sum_{l \neq 0} \frac{1}{\abs{lq_0}^\alpha}
\le \frac{1}{q_0^\alpha}.
\end{equation} 
If $\abs{\overline{qp}_0} \le q_0/2 \le k$ then
\begin{equation}
\sum_{\{l\,:\,\abs{\overline{qp}_0 + lq_0}> k\}} \frac{1}{\abs{\overline{qp}_0 + lq_0}^\alpha}
\lesssim \sum_{l \ge k/q_0} \frac{1}{\abs{lq_0}^\alpha}
\lesssim \frac{1}{q_0 k^{\alpha - 1}}.
\end{equation}
Now we have to compute the contribution of each case
to the sum in \eqref{eq:Upper_e_2}

The contribution of the case $\abs{\overline{qp}_0} > k$ is less than
\begin{equation} \label{eq:e_contribution_k_smallest}
A_1 := \frac{q_0^\alpha}{N^\alpha}\sum_{0 \le k < q_0/2}\;
\sum_{\{q\,:\,k < 2c_1q_0q/N \le k + 1\}}
	\ind_{\{\abs{\overline{qp}_0} > k\}}(q)\frac{q^{\alpha - s}}{\abs{\overline{qp}_0}^\alpha}. 
\end{equation}  
The last sum in $q$ runs over an interval of length
$N/(2c_1q_0) \ge q_0$ if $c_0 \ll 1$
(recall that $q_0 \le c_0\sqrt{N}$), so
we can break it into a number $\simeq N/(2c_1q_0^2)$ of 
blocks $K$ of length $q_0$ so that
\begin{equation}
\sum_{\{q\,:\,k < 2c_1q_0q/N \le k + 1\}}
	\ind_{\{\abs{\overline{qp}_0} > k\}}(q)\frac{q^{\alpha - s}}{\abs{\overline{qp}_0}^\alpha} \le 
\sum_{K}\sum_{q \in K} \ind_{\{\abs{\overline{qp}_0} > k\}}(q)
\frac{q^{\alpha - s}}{\abs{\overline{qp}_0}^\alpha};
\end{equation}
see Figure \ref{fig:blocks}.

\begin{figure}[h]
\includegraphics[width=0.75\linewidth]{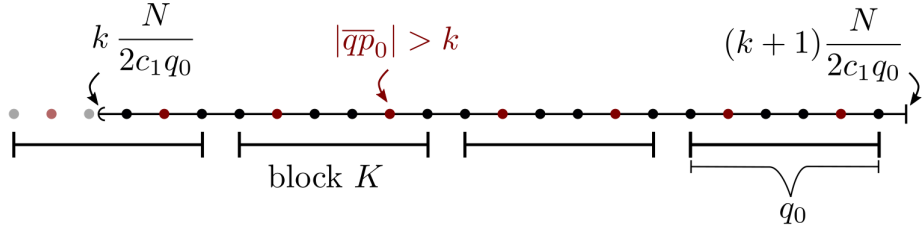}
\caption{Blocks $K$ of length $q_0$}
 \label{fig:blocks}
\end{figure}

\noindent Since $(p_0, q_0) = 1$, then $\overline{qp}_0$ runs
over all residues $r$ mod $q_0$, so
we may bound this case as
\begin{align}
\sum_{\{q\,:\,k < 2c_1q_0q/N \le k + 1\}}
	\ind_{\{\abs{\overline{qp}_0} > k\}}(q)\frac{q^{\alpha - s}}{\abs{\overline{qp}_0}^\alpha} &\lesssim
\Big((k+1)\frac{N}{c_1q_0}\Big)^{\alpha - s}\sum_K \sum_{\substack{r\in K\\\abs{r} > k}}
\frac{1}{\abs{r}^\alpha} \\
&\lesssim 
\frac{1}{q_0(k+1)^{s - 1}}\Big(\frac{N}{c_1q_0}\Big)^{\alpha + 1 - s}
\end{align}
We replace it in \eqref{eq:e_contribution_k_smallest} so that
\begin{equation}
A_1 \lesssim \frac{q_0^\alpha}{N^\alpha}
	\sum_{0 \le k < q_0/2}\;
	\frac{1}{q_0(k+1)^{s - 1}}\Big(\frac{N}{c_1q_0}\Big)^{\alpha + 1 - s}\lesssim \frac{q_0^{s- 2}}{N^{s - 1}}. \label{eq:A1}
\end{equation}

The contribution of the case $\abs{\overline{qp}_0} \le k < q_0/2$ is less than
\begin{align}
A_2 &:=
\frac{q_0^\alpha}{N^\alpha}
	\sum_{0 \le k < q_0/2}\sum_{\{q\,:\,k < 2c_1q_0q/N \le k + 1\}} 
	\ind_{\{\abs{\overline{qp}_0} \le k\}}(q)\frac{q^{\alpha - s}}{q_0^\alpha}	\\
&\lesssim 
\frac{1}{N^\alpha}
	\sum_{0 \le k < q_0/2}\Big[(k+1)\frac{N}{q_0}\Big]^{\alpha - s}
	\sum_{\{q\,:\,k < 2c_1q_0q/N \le k + 1\}}
	\ind_{\{\abs{\overline{qp}_0} \le k\}}(q).
\end{align}
We estimate the last sum in $q$ as before, 
breaking the sum into blocks of length $q_0$, so that
\begin{align}
A_2 &\lesssim
\frac{1}{N^sq_0^{\alpha - s}}
\sum_{0 \le k < q_0/2}(k+1)^{\alpha - s}
\sum_K \sum_{\substack{r\in K\\\abs{r} \le k}}
\ind_{\{\abs{\overline{qp}_0} \le k\}}(q) \\
&\lesssim
\frac{1}{N^sq_0^{\alpha - s}}
\sum_{0 \le k < q_0/2}(k+1)^{\alpha - s}(k+1)
\frac{N}{2c_1q_0^2} \\
&\lesssim \frac{1}{N^{s - 1}}.
\label{eq:A2}
\end{align}

The contribution of the case $k \ge q_0/2$ is less than
\begin{align}
A_3 &:= \frac{q_0^\alpha}{N^\alpha}\sum_{k \ge q_0/2}\;\sum_{\{q\,:\,k < 2c_1q_0q/N \le k + 1\}}\;
	\frac{q^{\alpha - s}}{q_0k^{\alpha - 1}} \\
&\lesssim
\frac{q_0^\alpha}{N^\alpha}\cdot
\frac{N^{\alpha + 1 - s}}{q_0^{\alpha + 2 - s}}
	\sum_{k \ge q_0/2}\frac{1}{k^{s - 1}} \\
&\lesssim \frac{1}{N^{s - 1}}.
\label{eq:A3}
\end{align} 

We sum up all the contributions 
\eqref{eq:A1}, \eqref{eq:A2} and \eqref{eq:A3} to 
the error term \eqref{eq:Upper_e_2} to find out
\begin{equation}
\abs{e(x)} \lesssim A_1 + A_2 + A_3 \lesssim 
\frac{q_0^{s- 2}}{N^{s - 1}},
\qquad \textrm{for } x \in I_{p_0/q_0},
\end{equation}
which is what we wanted.
\end{proof}

After this lemma, for each $p_0/q_0$ with $q_0 \le c_0\sqrt{N}$,
we can split $P_Nh_{\p,\delta}$ as
\begin{equation}
P_N h_{\textrm{p},\delta}(x) 
= \sum_{\substack{p/q\, \in \,2I_{p_0/q_0} \\ q \le c_0\sqrt{N}}}
	\frac{a_{q,\delta}}{q^s}\phi(N(x - p/q)) + 
	\sum_{\substack{p/q\, \in \,2I_{p_0/q_0} \\ q > c_0\sqrt{N}}}
	\frac{a_{q,\delta}}{q^s}\phi(N(x - p/q))
	+ e(x),
	\qquad\textrm{if } x \in I_{p_0/q_0}.
\end{equation}
Since the only fraction $p/q\in 2I_{p_0/q_0}$ with
$q \le c_0\sqrt{N}$, for $c_0 \ll 1$, is $p_0/q_0$ itself,
then we can write this decomposition as
\begin{align}
P_N h_{\textrm{p},\delta}(x) 
&= \frac{a_{q_0,\delta}}{q_0^s}\phi(N(x - p_0/q_0)) + 
	\sum_{\substack{p/q\, \in \,2I_{p_0/q_0} \\ q > c_0\sqrt{N}}}
	\frac{a_{q,\delta}}{q^s}\phi(N(x - p/q))
	+ e(x),
	\qquad\textrm{if } x \in I_{p_0/q_0}, \\
&=: \vc{M}(x) + \vc{E}(x) + e(x).
\label{eq:lowerB_decomposition}
\end{align}

\subsubsection{The range $p > 2/s$}

We estimate the $L^p$ norm of $\vc{M}$ by
integrating in the interval $I_0$,
that is, $p_0/q_0 = 0$, so
\begin{equation}
\norm{\vc{M}}_{L^p(I_0)} \gtrsim 
	N^{-1/p}.
\end{equation}
The $L^p$ norm of $\vc{E}$ is
\begin{equation}
\norm{\vc{E}}_{L^p(I_0)} \lesssim
	\frac{1}{N^{1/p}}\sum_{\substack{p/q \in 2I_0 \\ q > c_0\sqrt{N}}}
		\frac{1}{q^s}.
\end{equation}
Since $\abs{p/q}\le 2c_1/N$ then necessarily
$q \ge N/(2c_1) > c_0\sqrt{N}$, and
the number of fractions with denominator $q$ in
$2I_0$ is $\le 2c_1q/N$, so
\begin{equation}
\norm{\vc{E}}_{L^p(I_0)} \lesssim
\frac{1}{N^{1/p}}\sum_{q \ge N/(2c_1)}\frac{1}{Nq^{s - 1}} \lesssim N^{-1/p - s + 1}.
\end{equation}
By Lemma~\ref{thm:Tails_Negligible} we have
\begin{equation}
\norm{e}_{L^p(I_0)} \lesssim N^{-1/p -s + 1}.
\end{equation}
This leads us to the conclusion
\begin{equation}
\label{eq:lowerpg1}
\norm{P_N h_{\textrm{p},\delta}}_p \ge 
\norm{P_N h_{\textrm{p},\delta}}_{L^p(I_0)} \gtrsim
N^{-1/p},
\end{equation}
which proves the lower bound in Theorem~\ref{thm:Lp_norm_N} for the range $p>2/s$.
 
\subsubsection{The range $0<p \le 2/s$}

The lower bound will be estimated by integrating
$P_Nh_{\p,\delta}$ over
\begin{equation} \label{eq:Set_I_Lower}
U := \bigcup_{\substack{p/q \\ q \le c_0\sqrt{N}}} I_{p/q}.
\end{equation} 
Since the intervals are pairwise disjoint (see \eqref{eq:condc0c1})
\begin{equation}
\label{eq:measU}
|U|= \sum_{\substack{p/q \\ q \le c_0\sqrt{N}}} |I_{p/q}|\le \frac{2c_1}{N}\sum_{\substack{p/q \\ q \le c_0\sqrt{N}}} 1\le  \frac{2c_1}{N}\sum_{ q \le c_0\sqrt{N}} \varphi(q)\le  \frac{2c_1}{N} \sum_{ q \le c_0\sqrt{N}}q\le 2c_0^2c_1.
\end{equation}

By Lemma~\ref{thm:Tails_Negligible} the $L^p$ norm of $e$ is small
\begin{equation} \label{eq:Lp_e}
\norm{e}_{L^p(U)}^p \lesssim 
\sum_{\substack{p/q \\ q \le c_0\sqrt{N}}}
\frac{q^{p(s- 2)}}{N^{p(s - 1)}}\abs{I_{p/q}}
\lesssim \frac{1}{N^{p(s - 1) + 1}}
	\sum_{q \le c_0\sqrt{N}}q^{p(s - 2) + 1} 
\lesssim 
c_0^{p(s - 2) + 2}N^{-ps/2}.
\end{equation}

In the decomposition \eqref{eq:lowerB_decomposition},
the $L^p$ norm of the main term is
\begin{align}
\norm{\vc{M}}_{L^p(U)} &\gtrsim
	\Big(\sum_{\substack{p/q \\ q \le c_0\sqrt{N}}}
	\int_{x \in I_{p/q}}
		\frac{1}{q^{ps}}\,\abs{\phi(N(x - p/q))}^p\,dx\Big)^{1/p}  \\
&\gtrsim \Big(\frac{1}{N}\sum_{q \le c_0\sqrt{N}}
	\frac{\varphi(q)}{q^{ps}}\Big)^{1/p},
\label{eq:Preliminary_lower_bound}
\end{align}
by using the properties of $\phi$.
We estimate the last sum in the next lemma.

\begin{lem}\label{thm:lower_Bound_Totient}
Let $0 < \alpha \le 2$ and $M \gg 1$. Then,
\begin{equation} \label{eq:thm:Totient_sum}
\sum_{1 \le q \le M}\frac{\varphi(q)}{q^\alpha} \gtrsim
\begin{cases}
\log M & \textrm{if } \alpha = 2, \\
M^{2 - \alpha} & \textrm{if } 0 < \alpha < 2.
\end{cases}
\end{equation}
\end{lem}

\begin{proof}
Recall the identity
\begin{equation} \label{eq:Totient_Identity}
\varphi(q) = q\sum_{d\mid q}\frac{\mu(d)}{d},
\end{equation}
where $\mu$ is the M\"obius function;
see \cite[Section 16.3]{hardyWright2008}.
 
Let $M_0 \gg 1$ and 
replace \eqref{eq:Totient_Identity} into
the left-hand side of \eqref{eq:thm:Totient_sum} so that
\begin{align}
\sum_{1 \le q \le M}\frac{\varphi(q)}{q^\alpha} 
	&\ge
 	\sum_{M_0 \le q \le M}\frac{\varphi(q)}{q^\alpha} \\
	&\simeq \sum_{M_0 \le m \le M}
	\frac{1}{m^{\alpha - 1}}
		\sum_{m/2 \le q \le m}
		\sum_{d\mid q} \frac{\mu(d)}{d},
\label{eq:sum_totient_dyadic}
\end{align}
where $m \in 2^\N$. 
For each dyadic block we have
\begin{align}
\sum_{m/2 \le q \le m}
		\sum_{d\mid q} \frac{\mu(d)}{d}
&= \sum_{d \ge 1} \frac{\mu(d)}{d} 
	\sum_{m/2 \le q \le m} 
		\ind_{d\mid q}(q) \\
	&= \sum_{1 \le d \le m} \frac{\mu(d)}{d}
		\sum_{m/(2d) \le k \le m/d} 1 \\
	&= \sum_{1 \le d \le m} 
	\frac{\mu(d)}{d}\Big(\frac{m}{2d}\Big) + 
		\sum_{1 \le d \le m} 
	\frac{\mu(d)}{d}
		\Big( - \frac{m}{2d} + \sum_{m/(2d) \le k \le m/d} 1\Big) 
\end{align}
Since $\abs{m/(2d) - \abs{\{k \mid m/(2d) \le k \le m/d\}}} \lesssim 1$
and $\abs{\mu(d)} \le 1$, then
\begin{align}
\sum_{m/2 \le q \le m}
	\sum_{d\mid q} \frac{\mu(d)}{d}
	&= \frac{m}{2}\sum_{1 \le d \le m} 
	\frac{\mu(d)}{d^2} + \BigO(\log m) \\
&= \frac{m}{2\zeta(2)} -
	\frac{m}{2}\sum_{d > m}\frac{\mu(d)}{d^2} + \BigO(\log m) \\
&= \frac{m}{2\zeta(2)} + \BigO(\log m).
\end{align}
Here, we used the identity
$\sum_{d \ge 1}\mu(d)/d^2 = 1/\zeta(2)$, where
$\zeta$ is the Riemann zeta function, see \cite[Theorem 287]{hardyWright2008}.

Going back to \eqref{eq:sum_totient_dyadic}, 
for $M_0 \gg 1$ we get
\begin{equation}
\sum_{1 \le q \le M}\frac{\varphi(q)}{q^\alpha} \gtrsim
	\sum_{M_0 \le m \le M}\frac{1}{m^{\alpha - 2}},
\end{equation}
which yields \eqref{eq:thm:Totient_sum}.
\end{proof} 

We apply Lemma~\ref{thm:lower_Bound_Totient}
to \eqref{eq:Preliminary_lower_bound} with $c_0 \ll 1$ to find out that for $N \gg 1$ we have
\begin{equation} \label{eq:M_LBound}
\norm{\vc{M}}_p \gtrsim \begin{cases}
N^{-s/2}(\log c_0 N)^{1/p}, & \textrm{if } p = 2/s \\
c_0^{2/p - s}N^{-s/2}, & \textrm{if } p < 2/s. 
\end{cases}
\end{equation}

It remains to bound the error term 
\begin{equation} \label{eq:def_tilde_E}
\vc{E}(x) = 
\sum_{\substack{p/q \in 2I_{p_0/q_0} \\ q > c_0\sqrt{N}}}
		\frac{a_{q,\delta}}{q^s}\phi(N(x - p/q)),
\qquad \textrm{if } x \in I_{p_0/q_0}.
\end{equation} 
To compute the $L^p$ norm 
we use H\"older with exponent $r = 1/p$ so that, in view of \eqref{eq:measU},
\begin{equation} \label{eq:E_Lp}
\norm{\vc{E}}_{L^p(U)} \le 
\abs{U}^{1/(pr')}
	\norm{\vc{E}}_{L^1(U)} 
\lesssim c_0^{2/p - 2}\norm{\vc{E}}_{L^1(U)}.
\end{equation}
The $L^1$ norm is
\begin{equation} \label{eq:E_L1_A}
\norm{\vc{E}}_{L^1(U)} \lesssim 
	\sum_{\substack{p'/q' \\ q' \le c_0\sqrt{N}}}\;
	\sum_{\substack{p/q \,\in\, 2I_{p'/q'}\\ q > c_0\sqrt{N}}}
	\frac{1}{q^s}\abs{I_{p'/q'}} \lesssim
\frac{1}{N}	\sum_{\substack{p/q \\ q > c_0\sqrt{N}}}\frac{1}{q^s}
	\sum_{\substack{p'/q' \\ q' \le c_0\sqrt{N}}}\ind_{\{\abs{p/q - p'/q'} \le 2c_1/N\}}(p'/q').
\end{equation}
The last sum in $p'/q'$ is at most one
because of the restriction $q' \le c_0\sqrt{N}$.
Indeed, if there were at least two, 
say $p_1/q_1, p_2/q_2$ with $\abs{p/q - p_i/q_i} \le 2c_1/N$, 
then $1/(q_1q_2)\le \abs{p_1/q_2 - p_1/q_2}\le 4c_1/N$ and $1/(q_1q_2)\ge 1/(c_0^2 N)$, 
which implies $c_0>1/(2c_1^{1/2})$, 
contradiction to \eqref{eq:condc0c1}.
If the last sum in $p'/q'$ is not empty (which happens when $q >\sqrt{N}/(2c_0c_1)$), then 
necessarily $p'/q' \neq p/q$, which implies that
$1 \le 2c_1qq'/N$ or 
$q \ge N/(2c_1q')\ge \sqrt{N}/(2c_0c_1)$.
Hence, 
\begin{equation}  \label{eq:tildeE_L1_easy}
\norm{\vc{E}}_{L^1(U)} \lesssim 
	\frac{1}{N}	\sum_{\substack{p/q \\ q \ge \sqrt{N}/(2c_0c_1)}}\frac{1}{q^s}\lesssim 
	\frac{1}{N}	\sum_{q \ge \sqrt{N}/(2c_0c_1)}\frac{\varphi(q)}{q^s}
\lesssim c_0^{s-2}N^{-s/2}.
\end{equation}
In Appendix~\ref{app:counting_rationals} we prove a better upper bound,
but for the present purposes this is enough.
We get thus
\begin{equation} \label{eq:Lp_E_final}
\norm{\vc{E}}_{L^p(U)} \lesssim
	c_0^{2/p + s-4}N^{-s/2}.
\end{equation}

We can now conclude the lower bound.
We have
\begin{equation}
\norm{P_N h_{\textrm{p},\delta}}_{L^p(U)} \ge 
	\norm{\vc{M}}_{L^p(U)} - \norm{\vc{E}}_{L^p(U)}
	- \norm{e}_{L^p(U)},
\end{equation}
so, when $p < 2/s$, from \eqref{eq:M_LBound}, \eqref{eq:Lp_E_final}
and \eqref{eq:Lp_e} we get
\begin{equation}
\norm{P_N h_{\textrm{p},\delta}}_p \ge 
\norm{P_N h_{\textrm{p},\delta}}_{L^p(U)} \gtrsim
c_0^{2/p}(c_0^{-s} - Cc_0^s - Cc_0^{s - 4})N^{-s/2}
\gtrsim N^{-s/2} 
	\quad (\textrm{by } s>2).
\end{equation}
For the critical exponent $p = 2/s$ we get (observe that the only logarithmic term below comes from $\vc{M}$)
\begin{equation}
\norm{P_N h_{\textrm{p},\delta}}_p \ge 
\norm{P_N h_{\textrm{p},\delta}}_{L^p(U)} \gtrsim N^{-s/2}(\log c_0N)^{s/2},
\end{equation}
which concludes the proof of Theorem~\ref{thm:Lp_norm_N}
when $0 < p \le 2/s$. 

\appendix


\section{Numerical Simulations}
    \label{app:numerical}

For a given signal/function, we calculate its spectrum of singularities $D(h)$ 
numerically using the \textit{wavelet transform modulus maxima (WTMM) method} 
implemented in MATLAB using the Wavelab 850 toolbox \cite{buckheit1995wavelab}. 
With a suitable choice of a wavelet, through the wavelet coefficients, 
we compute the partition function, scaling exponent $\eta(p)$ and thus, 
$D(h)$ is estimated using the Legendre transform (see \cite{turiel2006} 
for their precise definition as they are different from the ones mentioned earlier). 
The input parameters consist of the signal $X$ with length $N=2^J$, 
the number of scales, range of parameters $p$ and $h$. 
Thus, for $X=H_\delta$ in Figure \ref{fig:HdeltaRNDF}.B, 
we choose $J=13$, $\alpha = 0.7, 0.9$, $p\in[-5,5]$, $h\in [h_{\min}, h_{\max}]$, 
where $h_{\min}=0$, $h_{\max}=1/\alpha$, that is, the support of $H_\delta$, 
and the wavelet used is the first derivative of a Gaussian. 
To further compare them quantitatively, 
we calculate the error as defined in \cite[(32)]{turiel2006} and 
obtain the values 0.0956 and 0.1166 for $\alpha=0.7$ and $\alpha=0.9$, respectively. 
We notice that these results are indeed comparable with 
the ones obtained in \cite{turiel2006} and can be reduced further with a larger $N$.

Next, we estimate $D(h)$ for $R$ and $R_M$ in the context of vortex filament equation. 
More precisely, we consider the input signal $X$ as 
the trajectory of the third component (without the vertical height) of the $M$-sided filament curve \cite{hozVega2014}.
With $p\in[-5,5]$, $h_{\min}=0.4$ and $h_{\max}=0.8$, 
we plot the Riemann's function $R$ and
$R_M$, for $M=3, 5, 8, 15$, in Figure \ref{fig:HdeltaRNDF}.A. 
The plots show the estimated values of $D(h)$
where its maximum value varies with $M$ and converges to that of $R$ (circled points). Indeed, for $M=15$, the agreement is remarkable
and deviations from the theoretical values (starred points) are a result of a numerical error,
which is minimum when the wavelet chosen is the second derivative of a Gaussian.
The support of $D(h)$ in each case is very close to 0.25.


\section{Counting rationals}
	\label{app:counting_rationals}

In the next proposition 
we improve the upper bound 
 $\norm{\vc{E}}_{L^1(U)} \lesssim c_0^{s-2}N^{-s/2}$ 
we proved in \eqref{eq:tildeE_L1_easy}.

\begin{prop}
Let $U$ be the set \eqref{eq:Set_I_Lower} and
$\vc{E}$ the function \eqref{eq:def_tilde_E}. 
Then,
\begin{equation}
\norm{\vc{E}}_{L^1(U)} \lesssim
	c_0^sN^{-s/2}. 
\end{equation}
\end{prop}

\begin{proof}
The $L^1$ norm is
\begin{equation} 
\norm{\vc{E}}_{L^1(U)} \lesssim 
\frac{1}{N}	\sum_{\substack{p/q \\ q > c_0\sqrt{N}}}\;
\sum_{\substack{p'/q' \\ q' \le c_0\sqrt{N}}}
	\frac{1}{q^s}
	\ind_{\{\abs{p/q - p'/q'} \le 2c_1/N\}}(p/q, p'/q').
\end{equation}
Now we break the ranges $q > c_0\sqrt{N}$ and
$q' \le c_0\sqrt{N}$ dyadically into
parameters $\lambda, \mu \in 2^\N$, respectively, so that
\begin{equation}
\norm{\vc{E}}_{L^1(U)} \lesssim \frac{1}{N}
	\sum_{\substack{\lambda < c_0\sqrt{N} \\ \mu \ge c_0\sqrt{N}}}
	\mu^{-s}\sum_{\substack{p/q,\; p'/q' \\ q \simeq \lambda,\; q' \simeq \mu}}\ind_{\{\abs{p/q - p'/q'} \le 2c_1/N\}}(p/q, p'/q').
\end{equation}
Equivalently, we have to count at most how many
pairs of rationals $(p/q, p'/q')$ satisfy
\begin{equation} \label{eq:fractions_couting}
0<\abs{q'p - qp'} \le 2c_1\frac{\lambda\mu}{N};
\end{equation}
for that, we use the arguments in \cite[Proposition~4.2]{Pierce2021}.

For $m \in \Z\setminus\{0\}$, 
the goal is to count how many representations 
has $m$ as $m = q'p - qp'$ with $0\le p < q$ and $0\le p' < q'$,
so let us fix $q$ and $q'$.
If $(q,q') = d$ then necessarily $m = d\tilde m$,
so let us clear out $d$ from the representation of $m$
and write $\tilde m = \tilde q'p - \tilde qp'$,
where $q' = d\tilde q'$, $q = d\tilde q$ and
$(\tilde q, \tilde q') = 1$.
Now assume that $\tilde q'p - \tilde qp' = \tilde q'r - \tilde qr'$,
or after reordering $\tilde q'(p - r) = \tilde q(p' - r')$.
This implies that $\tilde q \mid (p - r)$ so, for some $l\in \Z$, we have
$0\le r = p + l\tilde q < q$ (recall that $0\le p<q$ and $q=d\tilde{q}$) and then the number of
different $r$'s is at most $d$.
In conclusion, for every $m$ divisible by $d$ there
are $d$ representations $m = q'p - qp'$ with $0\le p < q$ and $0\le p' < q'$.

The above paragraph shows that,
for fixed $q$ and $q'$, 
the number of choice of pairs $(p,p')$ satisfying \eqref{eq:fractions_couting}
is $\le 2c_1\lambda\mu/N$,
so the total number of fractions satisfying \eqref{eq:fractions_couting} 
is $\lesssim 2c_1\lambda^2\mu^2/N$.
Notice that the collection of
representations is empty unless
$\lambda\mu \ge N/(2c_1)$, so
\begin{align}
\norm{\vc{E}}_{L^1(U)} &\lesssim \frac{1}{N^2}
	\sum_{\substack{\lambda < c_0\sqrt{N} \\ \mu \ge c_0\sqrt{N}}}
	\mu^{2 - s}\lambda^2\ind_{\{\lambda\mu \ge N/(2c_1)\}}(\lambda, \mu) \\
&= \frac{1}{N^2}
	\sum_{\lambda < c_0\sqrt{N}} \lambda^2
	\sum_{\mu \ge N/(2c_1\lambda)}\mu^{2 - s} 
\qquad (\textrm{by } N/(2c_1\lambda) \ge c_0\sqrt{N}
\textrm{ for } c_0 \ll 1)\\
&\lesssim \frac{1}{N^s}
	\sum_{\lambda < c_0\sqrt{N}} \lambda^s
\lesssim c_0^sN^{-s/2},
\end{align}
where we used that $\mu$ and $\lambda$ are dyadic. The proof is completed.
\end{proof}

\bibliographystyle{acm}
\bibliography{intermittency}

\end{document}